\documentclass{amsart}

\usepackage[utf8]{inputenc}
\usepackage{mathrsfs}
\usepackage{amsmath,amssymb,amsfonts,amsthm}
\usepackage{amsthm}
\usepackage{enumerate}
\usepackage{amssymb}
\usepackage[all]{xy}
\usepackage{subcaption}
\usepackage[hidelinks]{hyperref}
\usepackage{latexsym}
\usepackage{tikz-cd}
\usepackage{comment}
\usepackage{yfonts} 
\usepackage{upgreek}
\usepackage{mathtools}
\usetikzlibrary{matrix, arrows}
\usetikzlibrary{decorations.markings}
\usepackage{rotating}
\usepackage{hyperref}
\usepackage{fancyhdr}
\usepackage{tabularx}

\newtheorem{thm}{Theorem}[section]
\newtheorem*{thm*}{Theorem}
\newtheorem{lem}[thm]{Lemma}
\newtheorem{prop}[thm]{Proposition}
\newtheorem{cor}[thm]{Corollary}

\theoremstyle{definition}

\newtheorem{exam}[thm]{Example}
\theoremstyle{remark}
\newtheorem{rmk}[thm]{Remark}

\numberwithin{equation}{section}
\newcommand{\OO}{\mathcal{O}}

\newcommand{\LL}{\mathcal{L}}

\DeclareMathOperator{\Spec}{Spec}

\DeclareMathOperator{\Hom}{Hom}

\DeclareMathOperator{\EExt}{\mathscr{E}\textit{\kern -1pt {xt}}\,}

\DeclareMathOperator{\Cone}{Cone}
\DeclareMathOperator{\Conv}{Conv}

\DeclareMathOperator{\CDiv}{CDiv}
\DeclareMathOperator{\Div}{Div}
\DeclareMathOperator{\ddiv}{div}
\DeclareMathOperator{\Cl}{Cl}
\DeclareMathOperator{\Pic}{Pic}

\DeclareMathOperator{\codim}{codim}
\DeclareMathOperator{\PGL}{PGL}

\begin{document}

\title{Cohomological vanishing on blown-up projective spaces}

\author{Marco Flores}

\thanks{The author acknowledges support from the Deutsche Forschungsgemeinschaft (DFG, German Research Foundation) under Germany's Excellence Strategy – The Berlin Mathematics Research Center MATH+ (EXC-2046/1, project ID: 390685689).}

\begin{abstract}
By utilizing elementary techniques from toric geometry, we prove sharp cohomological vanishing results for line bundles defined on the blow-up of projective space $\mathbb{P}^n$ at no more than $n+1$ points.
\end{abstract}

\maketitle

\section{Introduction}

Cohomological vanishing theorems are central to modern research in algebraic geometry, particularly within the subject of positivity. A prototypical example of such a theorem is Kodaira's vanishing theorem, which states that if $A$ is an ample divisor on a smooth complex projective variety $X$, then \[H^i(X,\OO_X(K_X+A))=0\] for all $i>0$. Following the original appearance of Kodaira's vanishing theorem in \cite{Kod53}, multiple generalizations of it and related theorems have come into light, the most prominent of which is perhaps the Kawamata--Viehweg vanishing theorem. This theorem was discovered independently by Kawamata \cite{Kaw82} and Viehweg \cite{Vie82}, and it states that Kodaira's vanishing theorem remains true if we only require the divisor $A$ to be big and nef. A rather rich collection of important cohomological vanishing theorems can be found in \cite[Chapter 4]{Laz04}.

In the realm of toric geometry, stronger cohomological vanishing results can be obtained, such as those presented in \cite[Sections 9.2 and 9.3]{CLS11}. Furthermore, there is a basic result in toric geometry (see Theorem \ref{torcoh} below) which compares the sheaf cohomology of a line bundle associated to a toric divisor on a toric variety, with the singular cohomology of a certain subset of Euclidean space, which is constructed following combinatorial data obtained from the toric variety and its toric divisor. This allows one to make very explicit computations of the sheaf cohomology groups of a line bundle associated to a toric divisor on a toric variety.

The aim of this short note is to show that we can apply Theorem \ref{torcoh} to the case of line bundles defined on the blow-up of projective space $\mathbb{P}^n$ at no more than $n+1$ points, in order to obtain sharp cohomological vanishing results. Here, we emphasize the sharpness of the results, as this is usually an unavailable perk in cohomological vanishing theorems in algebraic geometry. In the case of a blow-up at a single point, our result is the following.

\begin{thm*}[Corollary \ref{char1}]
    Let $\pi\colon X\to\mathbb{P}^n$ be a blow-up at one point, with exceptional divisor $E$, and let $a,b\in\mathbb{Z}$. Then
    \[H^1(X,\OO_X(-aE)\otimes\pi^*\OO_{\mathbb{P}^n}(b))=0 \iff a\leq 0 \textnormal{ or } a\leq b+1.\]
\end{thm*}

With some adaptations, our arguments can be extended to include the cases of blow-ups of $\mathbb{P}^n$ at multiple points, but no more than $n+1$. In this case, our result is as follows.

\begin{thm*}[Theorem \ref{mainthmsev}]
    Let $X\to\mathbb{P}^n$ be a blow-up at $q+1$ points, where $1\leq q\leq n$, with exceptional prime divisors $E_0,\dots, E_q$. Let $a_0,\dots,a_q,b\in\mathbb{Z}$. Then,
    \[H^1\big(X,\OO_X\big(-\sum_{i=0}^qa_iE_i\big)\otimes\pi^*\OO_{\mathbb{P}^n}(b)\big)=0\] if and only if the inequalities
    \[
        a_i+a_j\leq b+1 \ \ \forall i,j\in\{0,\dots,q\}, i\neq j
    \]are satisfied, unless there is exactly one element $a_k\in\{a_0,\dots,a_q\}$ that is positive, in which case we also require $a_k\leq b+1$.
\end{thm*}

Even though our methods and results are elementary, we believe it is interesting that we have been able to produce such characterizations, given that most expositions on this topic focus only on sufficient conditions on an invertible sheaf for it to have vanishing cohomology (cf. \cite[Theorem~1.3]{DP23}). To the best of our knowledge, the necessity of the conditions for cohomological vanishing in Corollary \ref{char1} and Theorem \ref{mainthmsev} had not previously been observed.

The organization of this paper is as follows. In Section~\ref{overview}, we recall some basic definitions and results in toric geometry over the complex numbers. In Section~\ref{onept}, we carry out the elementary exercise of determining the ample cone of a blow-up of $\mathbb{P}^n$ at one point, thence readily obtaining sufficient conditions for cohomological vanishing of invertible sheaves on this blown-up projective space, in all positive degrees, via the Kodaira vanishing theorem. We then prove Corollary \ref{char1}, which is our sharp characterization of cohomological vanishing in degree 1. Finally, in Section~\ref{sevpts}, we adapt the arguments given in the proof of Corollary \ref{char1} in order to treat the case of a blow-up of $\mathbb{P}^n$ at multiple points, not more than $n+1$, resulting in the proof of Theorem \ref{mainthmsev}.

\subsection*{Acknowledgements} We deeply thank Jürg Kramer, for his suggestion to investigate the matters present in this paper, as well as for several useful discussions.

\section{Overview on the basics of toric geometry}\label{overview}
In the present section we recall notation, basic definitions and some theorems from the classical theory of complex toric varieties, with an emphasis on sheaf cohomology and positivity. Our reference for this material is \cite{CLS11}.

\subsection*{Tori and their lattices}

Let $n\in\mathbb{N}$ be a natural number. An $n$-dimensional \textit{torus} is an algebraic group $T$ which is isomorphic to $(\mathbb{C}^*)^n$. A \textit{character} of a torus $T$ is a morphism $\chi\colon T\to\mathbb{C}^*$ of algebraic groups. The set $M$ of characters of an $n$-dimensional torus is a free abelian group of rank $n$, that is, it constitutes a lattice. A \textit{one-parameter subgroup} of a torus $T$ is a morphism $\lambda\colon \mathbb{C}^*\to T$ of algebraic groups. The set $N$ of one-parameter subgroups of an $n$-dimensional torus is also a lattice. Given a character $\chi\in M$ and a one-parameter subgroup $\lambda\in N$, the composition $\chi\circ\lambda\colon \mathbb{C}^*\to\mathbb{C}^*$ is a group homomorphism, so it is given by $t\mapsto t^k$ for some $k\in\mathbb{Z}$. By setting $\langle \chi, \lambda \rangle\coloneqq k$, we obtain a natural bilinear pairing $\langle\,,\rangle\colon  M\times N\to\mathbb{Z}$ which makes $M$ and $N$ dual lattices, that is, it induces isomorphisms $N\cong\Hom(M,\mathbb{Z})$ and $M\cong\Hom(N,\mathbb{Z})$. 

A $\textit{toric variety}$ is a variety $X$ containing a torus $T$ as a Zariski open subset, such that the action by entry-wise multiplication of $T$ on itself extends to an algebraic action of $T$ on $X$.

\subsection*{Cones and fans}
Let $M,N\cong\mathbb{Z}^n$ be dual lattices with a natural bilinear pairing $\langle\,,\rangle\colon  M\times N\to\mathbb{Z}$, which extends to a bilinear pairing of the real vector spaces $M_\mathbb{R}\coloneqq M\otimes_\mathbb{Z}\mathbb{R}$ and $N_\mathbb{R}\coloneqq N\otimes_\mathbb{Z}\mathbb{R}$. A \textit{convex polyhedral cone} in $N_\mathbb{R}$ is a set of the form \[\sigma=\Cone(S)=\Big\{\sum_{u\in S}\lambda_u u\mid\lambda_u\geq 0\Big\}\subseteq N_\mathbb{R},\] where $S\subseteq N_\mathbb{R}$ is finite. We say that $\sigma$ is \textit{generated by} $S$. If $S\subseteq N$, the polyhedral cone $\sigma=\Cone(S)$ is called \textit{rational}.

Let $\sigma$ be a convex polyhedral cone in $N_\mathbb{R}$. We define
\begin{align*}
    \sigma^\perp&\coloneqq\{m\in M_\mathbb{R}\mid\langle m,u\rangle = 0 \text{ for all } u\in\sigma\}\subseteq M_\mathbb{R},\\
    \sigma^\vee&\coloneqq\{m\in M_\mathbb{R}\mid\langle m,u\rangle \geq 0 \text{ for all } u\in\sigma\}\subseteq M_\mathbb{R},
\end{align*}
and we call $\sigma^\vee$ the \textit{dual cone of} $\sigma$. Given $m\in M_\mathbb{R}$, we define
   \[ H_m \coloneqq\{u\in N_\mathbb{R}\mid\langle m,u\rangle =0\}\subseteq N_\mathbb{R},\]
and
\[
    H_{m}^+ \coloneqq\{u\in N_\mathbb{R}\mid\langle m,u\rangle \geq 0\}\subseteq N_\mathbb{R}. \]
We call $H_m$ a \textit{supporting hyperplane of} $\sigma$ if $\sigma\subseteq H_{m}^+$, which is equivalent to the condition that $m\in\sigma^\vee$. A \textit{face} $\tau$ \textit{of} $\sigma$ is the intersection of $\sigma$ with any of its supporting hyperplanes, written $\tau\preceq\sigma$. We denote by $\sigma(r)$ the set of all $r$-dimensional faces of $\sigma$. A face $\tau\neq\sigma$ is called a \textit{proper face}, written $\tau\prec\sigma$. We call $\sigma$ \textit{strongly convex} if $\{0\}\preceq\sigma$. From this point forward, we refer to strongly convex rational polyhedral cones simply as ``cones".

A one-dimensional face $\rho\in\sigma(1)$ of a cone $\sigma$ is called a $\textit{ray}$. Since $\sigma$ is rational, namely, it is generated by lattice vectors in $N$, the semigroup $\rho\cap N$ is generated by a unique element $u_\rho\in\rho\cap N$, which we call the \textit{ray generator of} $\rho$. A cone $\sigma$ is generated by the ray generators of its one-dimensional faces, which are hence called the \textit{minimal generators of} $\sigma$. We say $\sigma$ is \textit{smooth} if its minimal generators are part of a $\mathbb{Z}$-basis of $N$.

A \textit{fan} $\Sigma$ in $N_\mathbb{R}$ is a finite collection of cones in $N_\mathbb{R}$, such that: the face of any cone in $\Sigma$ is also in $\Sigma$, and any two cones in $\Sigma$ intersect along a common face. We denote by $\Sigma(r)$ the set of all $r$-dimensional cones of $\Sigma$. The \textit{support of} a fan $\Sigma$, denoted $\lvert\Sigma\rvert$, is defined as the union of all its cones. We say $\Sigma$ is \textit{complete} if $\lvert\Sigma\rvert=N_\mathbb{R}$, and \textit{smooth} if every cone in $\Sigma$ is smooth.

\subsection*{The toric variety associated to a fan}

Given a cone $\sigma\subseteq N_\mathbb{R}$, the lattice points $\sigma^\vee\cap M$ form a finitely generated semigroup. The associated semigroup algebra $\mathbb{C}[\sigma^\vee\cap M]$ is then an integral domain, finitely generated as a $\mathbb{C}$-algebra, and \[U_\sigma\coloneqq \Spec\mathbb{C}[\sigma^\vee\cap M]\] is an $n$-dimensional normal affine toric variety. Let $\tau=\sigma\cap H_m$ be a face of~$\sigma$, so $m\in\sigma^\vee\cap M$. Then $U_\tau$ is isomorphic to the spectrum of the localization $\mathbb{C}[\sigma^\vee\cap M]_m$, which we denote as $(U_\sigma)_\tau$. 

Let now $\Sigma$ be a fan. For every pair of cones $\sigma_1,\sigma_2\in\Sigma$, we then obtain isomorphisms \[(U_{\sigma_1})_{\sigma_2}\cong U_{\sigma_1\cap\sigma_2}\cong (U_{\sigma_2})_{\sigma_1}.\] By gluing the affine varieties $\{U_\sigma\}_{\sigma\in\Sigma}$ along the aforementioned isomorphisms, we obtain a normal separated toric variety $X_\Sigma$ whose torus is given by $T_N\coloneqq U_{\{0\}}=\Spec\mathbb{C}[M]$. Moreover, every normal separated toric variety arises from a fan in this way \cite[Corollary~3.1.8]{CLS11}. The variety $X_\Sigma$ is complete if and only if $\Sigma$ is a complete fan \cite[Theorem~3.4.6]{CLS11}, and it is smooth if and only if $\Sigma$ is a smooth fan \cite[Theorem~3.1.19]{CLS11}.

\begin{exam}
    Let $N=\mathbb{Z}^n$, with standard basis $\{e_1,\dots,e_n\}$. Set \[e_0\coloneqq-e_1-e_2-\cdots -e_n,\] and let $\Sigma$ be the fan in $N_\mathbb{R}=\mathbb{R}^n$ consisting of the cones generated by all proper subsets of $\{e_0,\dots,e_n\}$. Then $X_\Sigma\cong\mathbb{P}^n$.
\end{exam}

The points in the affine toric variety $U_\sigma$ associated to a cone $\sigma$ are in bijective correspondence with semigroup homomorphisms $\gamma\colon \sigma^\vee\cap M\to(\mathbb{C},\cdot)$. There is a distinguished point in $U_\sigma$, denoted by $\gamma_\sigma$, which corresponds to the semigroup homomorphism
\[ \sigma^\vee\cap M\to(\mathbb{C},\cdot)\] given by the assignment
\[m\mapsto \begin{cases}
      1 & \text{if}\ m\in\sigma^\perp\cap M, \\
      0 & \text{otherwise.}
    \end{cases}\]
The following fundamental theorem of toric geometry, known as the Orbit-Cone Correspondence, shows that a toric variety has a natural stratification by torus orbits, induced by the cones in its associated fan.

\begin{thm}\textnormal{\cite[Theorem~3.2.6]{CLS11}}
    Let $X_\Sigma$ be the toric variety associated to the fan $\Sigma$ in $N_\mathbb{R}$. For each cone $\sigma\in\Sigma$, we denote the $T_N$-orbit of the distinguished point $\gamma_\sigma$ in $X_\Sigma$ by $O(\sigma)$. Then:
    \begin{enumerate}
        \item[(i)] The map

        \[\Sigma\to\{T_N\textnormal{-orbits in } X_\Sigma\},\]

    given by the assignment $\sigma\mapsto O(\sigma)$,
        is a bijection.
        \item[(ii)] We have \[\codim (O(\sigma))=\dim(\sigma).\] 
        \item[(iii)] We have \[U_\sigma=\bigcup_{\tau\preceq\sigma}O(\tau).\]
        \item[(iv)] We have $\tau\preceq\sigma$ if and only if $O(\sigma)\subseteq\overline{O(\tau)}$. Moreover, the equality \[\overline{O(\tau)}=\bigcup_{\tau\preceq\sigma}O(\sigma)\] holds.
    \end{enumerate}
\end{thm}

\subsection*{Toric morphisms}

Let $N_1$ and $N_2$ be two lattices, $\Sigma_1$ a fan in $N_{1,\mathbb{R}}$, and $\Sigma_2$ a fan in $N_{2,\mathbb{R}}$. A $\textit{morphism of}$ $\textit{fans}$ $\Sigma_1\to\Sigma_2$ is a linear map $\overline{\phi}\colon N_{1,\mathbb{R}}\to N_{2,\mathbb{R}}$ such that it restricts to a homomorphism $N_1\to N_2$ and, for every cone $\sigma_1\in\Sigma_1$, there exists a cone $\sigma_2\in\Sigma_2$ such that $\overline{\phi}(\sigma_1)\subseteq\sigma_2$. If $X_1$ and $X_2$ are normal toric varieties, a morphism $\phi\colon X_1\to X_2$ is called $\textit{toric}$ if $\phi$ maps the torus $T_1\subseteq X_1$ into the torus $T_2\subseteq X_2$, and its restriction $T_1\to T_2$ is a group homomorphism. Any toric morphism $\phi\colon X_1\to X_2$ is $\textit{equivariant}$, meaning $\phi(t_1\cdot x_1)=\phi(t_1)\cdot\phi(x_1)$ for every $t_1\in T_1,x_1\in X_1$. 

A morphism of fans $\Sigma_1\to\Sigma_2$ induces a toric morphism $X_{\Sigma_1}\to X_{\Sigma_2}$ and, conversely, a toric morphism $X_{\Sigma_1}\to X_{\Sigma_2}$ induces a morphism of fans $\Sigma_1\to\Sigma_2$. This yields an equivalence of categories between the category of fans with morphisms of fans, and the category of normal toric varieties with toric morphisms.

\begin{exam}
    Given fans $\Sigma$ and $\Sigma'$ in $N_\mathbb{R}$, we say that $\Sigma'$ is a \textit{refinement of} $\Sigma$ if every cone of $\Sigma'$ is contained in a cone of $\Sigma$ and $\lvert\Sigma'\rvert=\lvert\Sigma\rvert$. In this case, the identity map of $N_\mathbb{R}$ is a morphism $\Sigma'\to\Sigma$, and the induced toric morphism $\phi\colon X_\Sigma'\to X_\Sigma$ is proper and birational.
\end{exam}

\begin{exam}\label{starsub}
    Let $\Sigma$ be a fan in $N_\mathbb{R}$ with $\dim (N_\mathbb{R})=n$. Let $\sigma=\Cone(u_1,\dots,u_n)$ be a smooth cone in $\Sigma$, that is, such that $\{u_1,\dots,u_n\}$ is a $\mathbb{Z}$-basis for $N$. Let $u_0=u_1+\cdots+u_n\in N$ and let $\Sigma'(\sigma)$ be the set of all cones generated by subsets of $\{u_0,\dots,u_n\}$ not containing $\{u_1,\dots,u_n\}$. Let $\rho_0\coloneqq\Cone(u_0)$ be the ray generated by $u_0$. Then, \[\Sigma^*(\sigma)\coloneqq(\Sigma\setminus\{\sigma\})\cup\Sigma'(\sigma)\] is a fan in $N_\mathbb{R}$, which is a refinement of $\Sigma$ and is called the \textit{star subdivision of} $\Sigma$ \textit{along} $\sigma$. The corresponding toric morphism $\pi\colon X_{\Sigma^*(\sigma)}\to X_\Sigma$ is the blow-up of $X_\Sigma$ at the distinguished point $\gamma_\sigma$, with exceptional divisor given by $\pi^{-1}(\gamma_\sigma)=\overline {O(\rho_0)}$.
\end{exam}

\subsection*{Toric divisors}
Let $X_\Sigma$ be the toric variety associated to a fan $\Sigma$ in $N_\mathbb{R}$, with $\dim (N_\mathbb{R})=n$. By the Orbit-Cone Correspondence, a ray $\rho\in\Sigma(1)$ corresponds to an $(n-1)$-dimensional orbit $O(\rho)$ whose closure $D_\rho\coloneqq \overline{O(\rho)}$ is a $T_N$-invariant Weil divisor on $X_\Sigma$. Divisors of the form $\sum_{\rho\in\Sigma(1)}a_\rho D_\rho$ with $a_\rho\in\mathbb{Z}$ are called \textit{toric divisors}. These are precisely the divisors on $X_\Sigma$ which are invariant under the torus action. Thus \[\Div_{T_N}(X_\Sigma)\coloneqq \bigoplus_{\rho\in\Sigma(1)}\mathbb{Z}D_\rho\subseteq\Div(X_\Sigma)\] is the group of $T_N$-invariant divisors on $X_\Sigma$. 

Let $u_\rho\in\rho\cap N$ be the ray generator of a ray $\rho\in\Sigma(1)$, and let $m\in M$ be a character of $T_N$. Then $m\colon T_N\to\mathbb{C}^*$ is a rational function on $X_\Sigma$, and its associated divisor is given by \[\ddiv(m)=\sum_{\rho\in\Sigma(1)}\langle m,u_\rho\rangle D_\rho\in\Div_{T_N}(X_\Sigma).\] We have an exact sequence

\begin{equation}\label{ses1}
    M\to\Div_{T_N}(X_\Sigma)\to\Cl(X_\Sigma)\to 0
\end{equation}
where the first map is given by $m\mapsto\ddiv(m)$ and $\Cl(X_\Sigma)$ denotes the class group of $X_\Sigma$. In particular, every divisor on $X_\Sigma$ is linearly equivalent to  a toric divisor. Moreover, the map given by $m\mapsto\ddiv(m)$ is injective if and only if the set $\{u_\rho\mid\rho\in\Sigma(1)\}$ spans $N_\mathbb{R}$.

Now let $D$ be a Cartier divisor on $X_\Sigma$. From the exact sequence \eqref{ses1}, $D$ is linearly equivalent to a toric divisor, which is itself Cartier. Let $\CDiv_{T_N}(X_\Sigma)\subseteq\Div_{T_N}(X_\Sigma)$ denote the subgroup of $T_N$-invariant Cartier divisors. Then, since $\ddiv(m)$ is Cartier for each $m\in M$, we have a further exact sequence \[M\to\CDiv_{T_N}(X_\Sigma)\to\Pic(X_\Sigma)\to 0\] where $\Pic(X_\Sigma)$ denotes the Picard group of $X_\Sigma$.

Let $\Sigma_{\max} \subseteq\Sigma$ be the set of maximal cones of $\Sigma$, that is, the cones in $\Sigma$ which are not proper subsets of another cone in $\Sigma$. Let $D=\sum_{\rho\in\Sigma(1)}a_\rho D_\rho$ be a toric divisor on $X_\Sigma$. Then, $D$ is Cartier if and only if for each $\sigma\in\Sigma_{\max}$ there exists a character $m_\sigma\in M$, such that $\langle m_\sigma, u_\rho\rangle=-a_\rho$ for every $\rho\in\sigma(1)$. When $D$ is Cartier, such a set $\{m_\sigma\mid\sigma\in\Sigma_{\max}\}$ is called the \textit{Cartier data of} $D$. Each $m_\sigma$ is unique modulo $\sigma^{\perp}\cap M$; in particular, if $\Sigma$ is complete, then $\Sigma_{\max}=\Sigma(n)$ and the Cartier data of $D$ is uniquely determined.

    Suppose that $\sigma_0\in\Sigma(n)\subseteq\Sigma_{\max}$ is a top-dimensional cone, and consider the subgroup

    \[A_{\sigma_0}\coloneqq\Big\{\sum_{\rho\in\Sigma(1)}a_\rho D_\rho\in\CDiv_{T_N}(X_\Sigma)\mid a_\rho=0 \textnormal{ for all }\rho\in\sigma_0(1)\Big\}\subseteq\CDiv_{T_N}(X_\Sigma).\] If $D=\sum_{\rho\in\Sigma(1)}a_\rho D_\rho$ is any toric Cartier divisor, whose Cartier data is given by $\{m_\sigma\mid\sigma\in\Sigma_{\max}\}$, then $D+\ddiv(m_{\sigma_0})\in A_{\sigma_0}$. This shows that the restriction $A_{\sigma_0}\to\Pic(X_\Sigma)$ of the natural surjective map $\CDiv_{T_N}(X_\Sigma)\to\Pic(X_\Sigma)$ remains surjective. But, moreover, this restriction map is also injective. Indeed, if $\ddiv(m)\in A_{\sigma_0}$ for some $m\in M$, then $\langle m,u_\rho\rangle=0$ for each $\rho\in\sigma_0(1)$, which implies that $m=0$ because $\sigma_0$ is top-dimensional. We have thus proven the following statement (cf. \cite[Exercise 6.1.6]{CLS11}).

\begin{prop}\label{toricpic}
    For each cone $\sigma_0\in\Sigma(n)$, the induced map
\[ \Big\{\sum_{\rho\in\Sigma(1)}a_\rho D_\rho\in\CDiv_{T_N}(X_\Sigma)\mid a_\rho=0 \textnormal{ for all }\rho\in\sigma_0(1)\Big\}\to \Pic(X_{\Sigma})\] is an isomorphism.
\end{prop}
We now recall that toric Cartier divisors behave well under pullbacks by a toric morphism. Let $N_1$ and $N_2$ be two lattices, with $\Sigma_1$ a fan in $N_{1,\mathbb{R}}$ and $\Sigma_2$ a fan in $N_{2,\mathbb{R}}$. Let $\phi\colon X_{\Sigma_1}\to X_{\Sigma_2}$ be a toric morphism, with $\overline{\phi}\colon N_{1,\mathbb{R}}\to N_{2,\mathbb{R}}$ its associated linear map, so that the restriction $\overline{\phi}\colon N_1\to N_2$ is a homomorphism inducing a dual homomorphism $\overline{\phi}^*\colon M_2\to M_1$. Then, the following result (cf. \cite[Proposition~6.2.7]{CLS11}) shows how to pullback a toric Cartier divisor on $X_{\Sigma_2}$ via $\phi$.

\begin{prop}\label{cartierpull}
    Let $D_2\in\CDiv_{T_{N_2}}(X_{\Sigma_2})$ with Cartier data $\{m_\sigma\mid\sigma\in(\Sigma_2)_{\max}\}$. For each $\sigma'\in(\Sigma_1)_{\max}$ we choose $\sigma\in(\Sigma_2)_{\max}$ such that $\overline{\phi}(\sigma')\subseteq\sigma$, and let \[m_{\sigma'}\coloneqq\overline{\phi}^*(m_\sigma).\] Then $\{m_{\sigma'}\mid\sigma'\in (\Sigma_1)_{\max}\}$ is the Cartier data of a toric Cartier divisor $D_1$ on $X_{\Sigma_1}$ satisfying $\OO_{X_{\Sigma_1}}(D_1)\cong\phi^*\OO_{X_{\Sigma_2}}(D_2)$.
\end{prop}

\subsection*{Positivity of toric divisors}
Recall that a sheaf $\mathcal{F}$ of $\mathcal{O}_X$-modules on a variety $X$ is \textit{generated by global sections} if there exists a set $\{s_i\}\subseteq H^0(X,\mathcal{F})$ such that, for every point $x\in X$, the stalk $\mathcal{F}_x$ is generated by the images of the global sections $\{s_i\}$. An invertible sheaf $\LL$ on $X$ is said to be \textit{ample} if for every coherent sheaf $\mathcal{F}$ on $X$, there is an integer $n_0>0$ (depending on $\mathcal{F}$) such that, for every $n\geq n_0$, the sheaf $\mathcal{F}\otimes\LL^n$ is generated by global sections. A Cartier divisor $D$ on $X$ is said to be: 

\begin{enumerate}
\item[(i)] \textit{nef}, if $D\cdot C\geq 0$ for every irreducible complete curve $C\subseteq X$,
    \item[(ii)] \textit{basepoint free}, if its associated invertible sheaf $\mathcal{O}_{X}(D)$ is generated by global sections, and
 \item[(iii)] \textit{ample}, if its associated invertible sheaf $\mathcal{O}_{X}(D)$ is ample.

\end{enumerate}

In the category of toric varieties, we can give a combinatorial characterization of nef, basepoint free and ample toric divisors, as follows. Let $X_\Sigma$ be a complete toric variety, and let $D=\sum_{\rho\in\Sigma(1)}a_\rho D_\rho$ be a toric Cartier divisor on $X_\Sigma$ with Cartier data $\{m_\sigma\mid\sigma\in\Sigma_{\max}\}$. Then, its \textit{support function} $\varphi_D\colon \lvert\Sigma\rvert\to\mathbb{R}$ is defined by \[\varphi_D(u)=\langle m_\sigma , u \rangle \] when $u\in\sigma\in\Sigma$. In particular, it satisfies $\varphi_D(u_\rho)=-a_{\rho}$ for each $\rho\in\Sigma(1)$. 

A cone $\tau\in\Sigma(n-1)$ is called a \textit{wall} if $\tau=\sigma\cap\sigma'$ for some $\sigma,\sigma'\in\Sigma(n)$. The support function $\varphi_D$ is known to be convex if, and only if, for every wall $\tau=\sigma\cap\sigma'$ there exists $u\in\sigma'\setminus\sigma$ such that the \textit{wall inequality} \[\varphi_D(u)\leq \langle m_\sigma ,u\rangle\] holds. Moreover, if $\varphi_D$ is convex, then the inequality $\varphi_D(u)\leq \langle m_\sigma ,u\rangle$ must hold for every $\sigma\in\Sigma(n), u\in\lvert\Sigma\rvert$ (cf. \cite[Lemma~6.1.5]{CLS11}).

Now, $\varphi_D$ is said to be \textit{strictly convex} if for every wall $\tau=\sigma\cap\sigma'$ there exists $u\in\sigma'\setminus\sigma$, such that the \textit{strict wall inequality} \[\varphi_D(u)<\langle m_\sigma ,u\rangle\] holds. Moreover, if $\varphi_D$ is strictly convex, then the inequality $\varphi_D(u) < \langle m_\sigma ,u\rangle$ must hold for every $\sigma\in\Sigma(n), u\in\lvert\Sigma\rvert\setminus\sigma$ (cf. \cite[Lemma~6.1.13]{CLS11}). These properties of the support function characterize nefness and ampleness of toric Cartier divisors.

\begin{thm}\label{nefcrit}
    Let $D$ be a toric Cartier divisor on a complete toric variety $X_\Sigma$. Then:
    \begin{enumerate}
    \item[(i)] $D$ is nef if and only if it is basepoint free, if and only if $\varphi_D$ is convex. 
        \item[(ii)] $D$ is ample if and only if $\varphi_D$ is strictly convex.
    \end{enumerate}
\end{thm}

\begin{proof}
    According to \cite[Theorem~6.1.7]{CLS11}, $D$ is basepoint free if and only if $\varphi_D$ is convex, while $D$ is basepoint free if and only if it is nef according to \cite[Theorem~6.3.12]{CLS11}. Part $(ii)$ is the content of \cite[Theorem~6.1.14]{CLS11}.
\end{proof}

\subsection*{Sheaf cohomology of toric varieties}
Let $D=\sum_{\rho\in\Sigma(1)}a_\rho D_\rho$ be a toric Cartier divisor on a toric variety $X_\Sigma$. Then, the group of global sections of $\OO_{X_\Sigma}(D)$ admits a grading \[H^0(X_\Sigma, \OO_{X_\Sigma}(D))=\bigoplus_{m\in M} H^0(X_\Sigma, \OO_{X_\sigma}(D))_m,\] where \begin{equation*} H^0(X_\Sigma, \OO_{X_\Sigma}(D))_m\coloneqq
    \begin{cases}
      \mathbb{C}\cdot m & \text{if}\ \langle m,u_\rho \rangle\geq -a_\rho \textnormal{ for every } \rho\in\Sigma(1),\\
      0 & \text{otherwise.}
    \end{cases} 
\end{equation*} This induces a natural grading of the cohomology groups \[H^i(X_\Sigma, \OO_{X_\Sigma}(D))=\bigoplus_{m\in M} H^i(X_\Sigma, \OO_{X_\Sigma}(D))_m\] for every $i\geq 0$.
    
    For each $m\in M$, we then consider the union of convex hulls \[V_{D,m}\coloneqq\bigcup_{\sigma\in\Sigma_{\max}}\Conv(u_\rho\mid\rho\in\sigma(1), \langle m,u_\rho\rangle<-a_\rho)\subseteq N_\mathbb{R}.\] Then, there is a comparison theorem between the sheaf cohomology of $\OO_{X_\Sigma}(D)$ and the singular cohomology of the sets $V_{D,m}\subseteq N_\mathbb{R}\cong\mathbb{R}^n$.

    \begin{thm}\textnormal{\cite[Theorem~9.1.3]{CLS11}}\label{torcoh}
        For each $m\in M$ and each $i\geq 0$, there is an isomorphism \[H^i(X_\Sigma,\OO_{X_\Sigma}(D))_m\cong \widetilde{H}^{i-1}(V_{D,m},\mathbb{C}),\] where $\widetilde{H}^{i-1}$ denotes reduced singular cohomology of degree $i-1$.
    \end{thm}

    As an application of Theorem~\ref{torcoh}, one has the following vanishing theorem (see \cite[Theorem~9.2.3]{CLS11} for a proof).

    \begin{thm}[Demazure vanishing]\label{Demazure}
        If $\lvert\Sigma\rvert$ is convex and $D$ is nef, then
        \[H^i(X_\Sigma,\OO_{X_\Sigma}(D))=0\] for all $i>0$.
    \end{thm}

    Another classic cohomological vanishing result is the Kodaira vanishing theorem, which states that if $X$ is a smooth projective variety and $A$ is an ample Cartier divisor on $X$, then \[H^i(X,\mathcal{O}_X(K_X+A))=0\] for all $i>0$, where $K_X$ denotes a canonical divisor on $X$ (see, e.g., \cite[Theorem~4.2.1]{Laz04} for a proof). Now, if $X_{\Sigma}$ is a toric variety, a torus invariant canonical divisor is given by \[K_{X_\Sigma}=-\sum_{\rho\in\Sigma(1)} D_\rho\] according to \cite[Theorem~8.2.3]{CLS11}. Hence, we may state the Kodaira vanishing theorem in the context of toric geometry as follows.

    \begin{thm}[Kodaira vanishing]
        Let $X_{\Sigma}$ be a smooth projective toric variety and let $A\in\Div_{T_N}(X_\Sigma)$ be an ample toric divisor. Then
        \[H^i\big(X_{\Sigma},\OO_{X_{\Sigma}}\big(A-\sum_{\rho\in\Sigma(1)}D_{\rho}\big)=0\] for all $i>0$.
    \end{thm}

\section{The blow-up of projective space at one point}\label{onept}
 In the present section, we utilize the tools of toric geometry, as recalled in Section~\ref{overview}, in order to determine the ample cone of a blow-up $X$ of $\mathbb{P}^n$ at one point, thence obtaining a cohomological vanishing result for invertible sheaves on $X$ via the Kodaira vanishing theorem. Moreover, we characterize the isomorphism classes of invertible sheaves on $X$ whose cohomology in degree 1 vanishes, with the use of Theorem~\ref{torcoh}.

Let $n\geq 3$ be an integer. Let $T\coloneqq (\mathbb{C}^\times)^n$ be the $n$-dimensional complex torus, and $\{e_1,\dots,e_n\}$ the standard basis of its lattice of one-parameter subgroups $N=\mathbb{Z}^n$. We set \[e_0\coloneqq  -e_1-e_2-\cdots -e_n\in\mathbb{Z}^n\] and we let $\Sigma$ be the fan in $N_\mathbb{R}=\mathbb{R}^n$ consisting of the cones generated by all proper subsets of $\{e_0,\dots,e_n\}$, that is, \begin{equation}\label{fanpn}
    \Sigma\coloneqq\{\Cone(S)\mid S\subsetneq\{e_0,\dots,e_n\}\}.
\end{equation} The associated toric variety is then $X_\Sigma\cong\mathbb{P}^n$. For each $i=0,\dots,n$ we let \[\sigma_i\coloneqq\Cone(e_0,\dots,\widehat{e}_i,\dots,e_n)\subseteq\mathbb{R}^n,\] where the hat notation $\widehat{e}_i$ means that the vector $e_i$ is excluded. The set of maximal cones of $\Sigma$ is then given by $\Sigma_{\max}=\{\sigma_0,\dots,\sigma_n\}$, while the set of ray generators of $\Sigma$ is equal to $\{e_0,\dots,e_n\}$. For each $i=0,\dots,n$, let $D_{i}$ denote the toric divisor in $\mathbb{P}^n$ associated to the ray generated by $e_i$.

We now let $\Sigma'$ denote the star subdivision of $\Sigma$ along $\sigma_0$, as in Example \ref{starsub}. Then $\Sigma'$ is a refinement of $\Sigma$, and the induced toric morphism 

\begin{equation}\label{toricblowup}
    \pi\colon X_{\Sigma'}\to\mathbb{P}^n
\end{equation} is the blow-up of $\mathbb{P}^n$ at the distinguished point $\gamma_{\sigma_0}$. 

\begin{rmk}
    Since $\PGL_{n+1}$ acts transitively on $\mathbb{P}^n$, making an appropriate change of coordinates shows that any blow-up $X\to\mathbb{P}^n$ at one point $p\in\mathbb{P}^n$ is of the form \eqref{toricblowup}.
\end{rmk}

For each $i=1,\dots,n$, we let \[\tau_i\coloneqq\Cone(u_0,e_1,\dots,\widehat{e}_i,\dots,e_n)\in\Sigma',\] where $u_0\coloneqq e_1+\cdots+e_n\in N$. We note that, by construction, the set of maximal cones of $\Sigma'$ is given by $\Sigma'_{\max}=\{\sigma_1,\dots,\sigma_n,\tau_1,\dots,\tau_n\}$, while the set of ray generators of $\Sigma'$ is equal to $\{u_0,e_0,e_1,\dots,e_n\}$. For each ray generator $u\in\{u_0,e_0,e_1,\dots,e_n\}$, let $D_u$ denote the toric divisor in $X_{\Sigma'}$ associated to the ray generated by $u$; in particular, $D_{u_0}$ is the exceptional divisor. We now compute an explicit formula for the pullback of toric divisors in $\mathbb{P}^n$ via the blow-up morphism $\pi\colon X_{\Sigma'}\to\mathbb{P}^n$.

\begin{rmk}
    Notice the notational choice that we have made. Since the ray generators $e_i$ appear in both the fan $\Sigma$ and the fan $\Sigma'$, we distinguish the associated toric divisors by writing $D_i$ for the divisor in $X_\Sigma\cong\mathbb{P}^n$, while writing $D_{e_i}$ for the divisor in $X_{\Sigma'}$.
\end{rmk}

\begin{lem}\label{pullback}
Let $D=\lambda_0D_0+\cdots+\lambda_nD_n$ be a toric divisor in $\mathbb{P}^n$. Then, the toric divisor \[D'\coloneqq\lambda_0D_{e_0}+\cdots +\lambda_n D_{e_n}+(\lambda_1+\cdots+\lambda_n)D_{u_0}\] in $X_{\Sigma'}$ satisfies $\OO_{X_{\Sigma'}}(D')\cong\pi^*\OO_{\mathbb{P}^n}(D)$.
\end{lem}

\begin{proof}
In $\mathbb{P}^n$ the groups of Weil and Cartier divisors coincide, so $D$ is Cartier and we can determine its Cartier data. Let $s\coloneqq\lambda_0+\cdots+\lambda_n$. For each $i=1,\dots,n$ we define
\[m_i\coloneqq(-\lambda_1,\dots,s-\lambda_i,\dots,-\lambda_n)\in M=\mathbb{Z}^n,\] where the entry equal to $s-\lambda_i$ is the $i$-th entry. We also define \[m_0\coloneqq(-\lambda_1,-\lambda_2,\dots,-\lambda_n)\in M=\mathbb{Z}^n.\] The set $\{m_0,\dots,m_n\}$ equals the Cartier data of $D$, since the equality $\langle m_i,e_j\rangle=-\lambda_j$ is satisfied whenever $i\neq j$. We note that, for every $i=1,\dots,n$, the smallest cone in $\Sigma$ containing $\tau_i\in\Sigma'$ is $\sigma_0$. Then, by Proposition~\ref{cartierpull} the Cartier divisor $D''$ on $X_{\Sigma'}$ whose Cartier data is given by \begin{align*}
    m_{\sigma_i}\coloneqq m_i ,\\
    m_{\tau_i}\coloneqq m_0
\end{align*} for $i=1,\dots,n$, satisfies $\OO_{X_{\Sigma'}}(D'')\cong\pi^*\OO_{\mathbb{P}^n}(D)$. Direct computation now shows that $D''=D'$. Indeed, we have \[\langle m_{\sigma_i}, e_j\rangle =-\lambda_j\] whenever $i\neq j$, and \[\langle m_{\tau_i}, u_0\rangle=-\lambda_1-\cdots-\lambda_n\] for all $i=1,\dots,n$, so \begin{align*}
    D''&=\lambda_0D_{e_0}+\cdots +\lambda_n D_{e_n}+(\lambda_1+\cdots+\lambda_n)D_{u_0}\\
    &= D'.
\end{align*} This finishes the proof.
\end{proof}

   We now wish to describe the Picard groups of $\mathbb{P}^n$ and $X_{\Sigma'}$ by choosing convenient sets of generators given by toric divisors. The group $\Pic(\mathbb{P}^n)\cong\mathbb{Z}$ is generated by the class of any hyperplane: we pick the toric divisor $D_{1}$ as representative and write $\Pic(\mathbb{P}^n)=\mathbb{Z} D_{1}$. Regarding a choice of generators for $\Pic(X_{\Sigma'})$, we note that Proposition~\ref{toricpic} applied to the cone $\sigma_1\in\Sigma'(n)$ gives
   \[\Pic(X_{\Sigma'})=\mathbb{Z}D_{u_0}+\mathbb{Z}D_{e_1}.\] 
   However, a natural choice of generators for $\Pic(X_{\Sigma'})$, using the language of invertible sheaves, is the pullback of $\OO_{\mathbb{P}^n}(1)\cong \OO_{\mathbb{P}^n}(D_1)$ via the blow-up $\pi$, and the ideal sheaf $\OO_{X_{\Sigma'}}(-D_{u_0})$ of the exceptional divisor. In order to write these generators more explicitly, we note that Lemma~\ref{pullback} shows that the toric divisor $D_{u_0}+D_{e_1}\in\Pic(X_{\Sigma'})$ satisfies \[\OO_{X_{\Sigma'}}(D_{u_0}+D_{e_1})\cong\pi^*\OO_{\mathbb{P}^n}(D_1).\] We thus choose $\{-D_{u_0}, D_{u_0}+D_{e_1}\}$ as $\mathbb{Z}$-basis for $\Pic(X_{\Sigma'})$.
   
\begin{thm}\label{nefness1}
    Let $\pi\colon X\to\mathbb{P}^n$ be a blow-up at one point with exceptional divisor $E$, and let $a,b\in\mathbb{Z}$. Then:
    \begin{enumerate}[(i)]
        \item $\OO_X(-aE)\otimes\pi^*\OO_{\mathbb{P}^n}(b)$ is nef if and only if $0\leq a\leq b$.
        \item $\OO_X(-aE)\otimes\pi^*\OO_{\mathbb{P}^n}(b)$ is ample if and only if $0< a< b$.
    \end{enumerate}    
\end{thm}

\begin{proof}
    We view $\mathbb{P}^n$ as the toric variety associated to the fan $\Sigma$ as defined in \eqref{fanpn}, and we identify $\pi$ with the toric morphism $\pi\colon X_{\Sigma'}\to\mathbb{P}^n$ induced by star subdivision as in \eqref{toricblowup}, by choosing coordinates so that the blown-up point matches the distinguished point $\gamma_{\sigma_0}$. Then, by Lemma~\ref{pullback} we have
    \[\OO_{X_{\Sigma'}}(-aE)\otimes\pi^*\OO_{\mathbb{P}^n}(b)\cong \OO_{X_{\Sigma'}}(-aD_{u_0}+b(D_{u_0}+D_{e_1})).\]
    Let us first compute, directly from the definition, the Cartier data \[\{m_\sigma\mid\sigma\in\Sigma'_{\max}\}=\{m_{\sigma_1},\dots,m_{\sigma_n},m_{\tau_1},\dots,m_{\tau_n}\}\subseteq M=\mathbb{Z}^n\] of the toric Cartier divisor $D\coloneqq -aD_{u_0}+b(D_{u_0}+D_{e_1})=(b-a)D_{u_0}+bD_{e_1}$.

    Let $i\in\{2,\dots,n\}$. The equalities \begin{align*}
        \langle m_{\sigma_1},e_i\rangle &= 0,\\
        \langle m_{\sigma_1},e_0\rangle &= 0
        \end{align*}
    imply that $m_{\sigma_1}=(0,\dots,0)$ is the zero character, while the equalities
    \begin{align*}
       \ \ \  \ \ \ \ \ \ \ \ \ \ \ \ \ \ \ \ \ \ \ \ \ \ \ \ \ \ \ \ \ \ \langle m_{\sigma_i},e_0\rangle &= 0,\\
        \langle m_{\sigma_i},e_1\rangle &= -b,\\
       \langle m_{\sigma_i},e_j\rangle &= 0 \qquad\forall\quad j\in\{2,\dots,n\}\setminus\{i\}
        \end{align*} imply that $m_{\sigma_i}=(-b,0,\dots,0,b,0,\dots,0)$, where $b$ occurs at the $i$-th entry. Furthermore, the equalities \begin{align*}
       \ \ \ \  \langle m_{\tau_1},e_i\rangle &= 0,\\
        \langle m_{\tau_1},u_0\rangle &= a-b
    \end{align*} imply that $m_{\tau_1}=(a-b,0,\dots,0)$, while the equalities
     \begin{align*}
        \ \ \  \ \ \ \ \ \ \ \ \ \ \ \ \ \ \ \ \ \ \ \ \ \ \ \ \ \ \ \ \ \ \langle m_{\tau_i},u_0\rangle &= a-b,\\
        \langle m_{\tau_i},e_1\rangle &= -b,\\
        \langle m_{\tau_i},e_j\rangle &= 0 \qquad\forall\quad j\in\{2,\dots,n\}\setminus\{i\}
    \end{align*} imply that $m_{\tau_i}=(-b,0,\dots,0,a,0,\dots,0)$, where $a$ occurs at the $i$-th entry.

    Suppose that $D$ is nef, so its support function $\varphi_D\colon \lvert\Sigma'\rvert\to\mathbb{R}$ is convex by Theorem~\ref{nefcrit}. In particular, $e_0\in\sigma_1\setminus\tau_1$ must satisfy the wall inequality \[0=\varphi_D(e_0)\leq\langle m_{\tau_1},e_0\rangle=b-a,\] and $e_2\in\tau_1\setminus\tau_2$ must satisfy the wall inequality \[0=\varphi_D(e_2)\leq\langle m_{\tau_2},e_2\rangle=a.\] Therefore $0\leq a\leq b$.

    Now suppose that $0\leq a\leq b$. Then all the wall inequalities are satisfied, namely:
    \begin{align*}
        -b&=\varphi_D(e_1)\leq\langle m_{\sigma_1},e_1\rangle=0,\\
        -b&=\varphi_D(e_1)\leq\langle m_{\tau_1},e_1\rangle=a-b,\\
        0&=\varphi_D(e_i)\leq\langle m_{\sigma_i},e_i\rangle=b \ \ \quad\qquad\forall\quad 2\leq i\leq n,\\
        0&=\varphi_D(e_i)\leq\langle m_{\tau_i},e_i\rangle=a \ \ \quad\qquad\forall\quad 2\leq i\leq n,\\
        0&=\varphi_D(e_0)\leq\langle m_{\tau_i},e_0\rangle=b-a \ \ \>  \quad\forall\quad 1\leq i\leq n,\\
        a-b&=\varphi_D(u_0)\leq\langle m_{\sigma_i},u_0\rangle=0 \ \quad\qquad\forall\quad 1\leq i\leq n,
    \end{align*} so $\varphi_D$ is convex and $D$ is nef. The corresponding statement about ampleness follows analogously, by arguing with strict convexity instead of convexity.
\end{proof}

We now turn our attention to the cohomology groups $H^i(X_{\Sigma'},\OO_{X_{\Sigma'}}(-aE)\otimes\pi^*\OO_{\mathbb{P}^n}(b))$, where $a,b\in\mathbb{Z}$. As an immediate consequence of Theorem~\ref{Demazure} and Theorem~\ref{nefness1}, if $0\leq a\leq b$, we then have \[H^i(X_{\Sigma'},\OO_{X_{\Sigma'}}(-aE)\otimes\pi^*\OO_{\mathbb{P}^n}(b))=0 \] for all $i>0$. However, we can obtain a better result by using the Kodaira vanishing theorem.

\begin{thm}\label{kodairaforpn1}
    Let $\pi\colon X\to\mathbb{P}^n$ be a blow-up at one point, with exceptional divisor $E$, and let $a,b\in\mathbb{Z}$ satisfying $0\leq a\leq b+1$. Then \[H^i(X,\OO_X(-aE)\otimes\pi^*\OO_{\mathbb{P}^n}(b))=0\] for all $i>0$.
\end{thm}

\begin{proof}
    As in the proof of Theorem~\ref{nefness1}, we view $\mathbb{P}^n$ as the toric variety associated to $\Sigma$, we identify $\pi$ with the toric morphism $\pi\colon X_{\Sigma'}\to\mathbb{P}^n$, and we set \[D\coloneqq -aD_{u_0}+b(D_{u_0}+D_{e_1}).\] Hence, we wish to show that $H^i(X_{\Sigma'},\OO_{X_{\Sigma'}}(D))=0$ for all $i>0$.

    According to \cite[Theorem~8.2.3]{CLS11}, the canonical sheaf $\omega_{X_{\Sigma'}}$ is given by
    \[\omega_{X_{\Sigma'}}\cong\OO_{X_{\Sigma'}}(-D_{u_0}-D_{e_0}-\cdots-D_{e_n}).\] We choose a different representative $K_{X_{\Sigma'}}$ of the canonical divisor class, as follows. We consider the character $m\coloneqq(-n,1,\dots,1)\in M$, and we set
    \begin{align*}
        K_{X_{\Sigma'}}\coloneqq&\ddiv(m)-D_{u_0}-D_{e_0}-\cdots-D_{e_n}\\
        =&(\langle m,u_0\rangle -1)D_{u_0}+(\langle m,e_0\rangle -1)D_{e_0}+\cdots + (\langle m,e_n\rangle -1)D_{e_n}\\
        =&-2D_{u_0}-(n+1)D_{e_1}\\
        =&(n-1)D_{u_0}-(n+1)(D_{u_0}+D_{e_1}).
    \end{align*}
   We then notice that the toric divisor
   \[D-K_{X_{\Sigma'}}=-(a+n-1)D_{u_0}+(b+n+1)(D_{u_0}+D_{e_1})\] is ample by Theorem~\ref{nefness1}, since
   \[0\leq a\leq b+1\implies 0<a+n-1<b+n+1.\] Therefore, $H^i(X_{\Sigma'},\OO_{X_{\Sigma'}}(D))=0$ for all $i>0$ by the Kodaira vanishing theorem.
\end{proof}

We will end this section by characterizing the cohomological vanishing of Theorem~\ref{kodairaforpn1} in the case $i=1$. For this, we will directly compute the dimension of the singular cohomology groups $\widetilde{H}^{0}(V_{D,m},\mathbb{C})$, and then apply Theorem~\ref{torcoh}.

\begin{lem}\label{pathconn}
    Let $D=\sum_{\rho\in\Sigma'(1)} a_\rho D_\rho$ be any toric divisor on $X_{\Sigma'}$, and let $m\in M$ be a character. If \[V_{D,m}=\bigcup_{\sigma\in\Sigma'_{\max}}\Conv(u_\rho\mid\rho\in\sigma(1), \langle m,u_\rho\rangle<-a_\rho)\subseteq\mathbb{R}^n\] is not path-connected, then $V_{D,m}=\{e_0,u_0\}$.
\end{lem}

\begin{proof}
    Recall that $\Sigma'_{\max}=\{\sigma_1,\dots,\sigma_n,\tau_1,\dots,\tau_n\}$, and $\Sigma'(1)$ consists of the rays generated by $u_0,e_0,\dots,e_n\in N=\mathbb{Z}^n$. Let us assume that $V_{D,m}$ is not path-connected. 
    
    Suppose that $e_i\in V_{D,m}$ for some $1\leq i\leq n$, and let $p\in V_{D,m}$ be any point. From the definition of $V_{D,m}$, we note that there is a line segment fully contained in $V_{D,m}$ connecting $p$ to some ray generator $v\in \{u_0,e_0,\dots,e_n\}$. Since $n\geq 3$, we can choose a maximal cone $\sigma\in\Sigma'_{\max}$ containing both $e_i$ and $v$. Indeed, if $1\leq j\leq n$, we choose $1\leq k\leq n$ with $k\neq i,j$ and we find $e_i,e_j\in\tau_k$, whereas for $v=e_0$ and $v=u_0$ we can choose $1\leq j\leq n$ with $j\neq i$ and we find $e_0,e_i\in\sigma_j$ and $u_0,e_i\in\tau_j$. Therefore, again by the definition of $V_{D,m}$, we deduce that the line segment between $e_i$ and $v$ is fully contained in $V_{D,m}$. This shows that every point $p\in V_{D,m}$ admits a path in $V_{D,m}$ to $e_i$, contradicting the assumption that $V_{D,m}$ is not path connected. Therefore $e_i\notin V_{D,m}$ for every $1\leq i\leq n$.
    
    Now, if only one of the points $e_0,u_0$ belongs to $V_{D,m}$, then $V_{D,m}$ consists of that single point, which is again a contradiction. Therefore $e_0,u_0\in V_{D,m}$. Since there is no cone in $\Sigma'_{\max}$ containing both $e_0$ and $u_0$, we conclude that $V_{D,m}=\{e_0,u_0\}$.
\end{proof}

\begin{thm}\label{cohp1}
Let $D=\lambda_0D_{e_0}+\cdots+\lambda_n D_{e_n}+\lambda_{n+1}D_{u_0}$ be a toric divisor on $X_{\Sigma'}$. Then, $\dim_{\mathbb{C}} H^1(X_{\Sigma'},\mathcal{O}_{X_{\Sigma'}}(D))$ is equal to the cardinality of the set
\[\big\{ (m_1,\dots,m_n)\in\mathbb{Z}^n\mid \lambda_0<m_1+\cdots+m_n<-\lambda_{n+1}, m_i\geq -\lambda_i \ \ \forall i=1,\dots,n \big\}.\]
\end{thm}

\begin{proof}
Since \[H^1(X_{\Sigma'},\OO_{X_{\Sigma'}}(D))=\bigoplus_{m\in M}H^1(X_{\Sigma'},\OO_{X_{\Sigma'}}(D))_m\] and \[H^1(X_{\Sigma'},\OO_{X_{\Sigma'}}(D))_m\cong\widetilde{H}^{0}(V_{D,m},\mathbb{C})\] for all $m\in M$ by Theorem~\ref{torcoh}, we have \[\dim_{\mathbb{C}}H^1(X_{\Sigma'},\OO_{X_{\Sigma'}}(D))=\sum_{m\in M}\dim_{\mathbb{C}} \widetilde{H}^{0}(V_{D,m},\mathbb{C}). \]
By Lemma~\ref{pathconn}, for each $m\in M$ we have \[\dim_{\mathbb{C}}\widetilde{H}^{0}(V_{D,m},\mathbb{C})\geq 1 \iff V_{D,m}=\{e_0,u_0\},\] and in this case $\dim_{\mathbb{C}}\widetilde{H}^{0}(V_{D,m},\mathbb{C})= 1$ because $\{e_0,u_0\}$ has two connected components. Therefore, \[\dim_{\mathbb{C}}H^1(X_{\Sigma'},\OO_{X_{\Sigma'}}(D))=\#\big\{m\in M\mid V_{D,m}=\{e_0,u_0\}\big\}.\] Let $m=(m_1,\dots,m_n)\in \mathbb{Z}^n$. By the definition of the set $V_{D,m}$, we have $V_{D,m}=\{e_0,u_0\}$ if and only if the inequalities
\begin{align*}
    -m_1-\cdots -m_n&=\langle m,e_0 \rangle <-\lambda_0, \\
    m_1+\cdots +m_n&=\langle m,u_0 \rangle <-\lambda_{n+1}, \\
    m_i&=\langle m,e_i \rangle \geq-\lambda_i \qquad \forall\quad i=1,\dots,n
\end{align*} are satisfied. Therefore, $\dim_{\mathbb{C}}H^1(X_{\Sigma'},\OO_{X_{\Sigma'}}(D))$ is equal to  \[\# \big\{ (m_1,\dots,m_n)\in\mathbb{Z}^n\mid \lambda_0<m_1+\cdots+m_n<-\lambda_{n+1} \textnormal{ and } m_i\geq -\lambda_i \ \ \forall i=1,\dots,n \big\}.\] This finishes the proof.
\end{proof}

    \begin{cor}\label{char1}
    Let $\pi\colon X\to\mathbb{P}^n$ be a blow-up at one point, with exceptional divisor $E$, and let $a,b\in\mathbb{Z}$. Then
    \[H^1(X,\OO_X(-aE)\otimes\pi^*\OO_{\mathbb{P}^n}(b))=0 \iff a\leq 0 \textnormal{ or } a\leq b+1.\]
\end{cor}

\begin{proof}
     As in the proof of Theorem~\ref{nefness1}, we view $\mathbb{P}^n$ as the toric variety associated to $\Sigma$, we identify $\pi$ with the toric morphism $\pi\colon X_{\Sigma'}\to\mathbb{P}^n$ and we set  \[D\coloneqq (b-a)D_{u_0}+bD_{e_1}.\] 
     
     Hence, we wish to show that $H^1(X_{\Sigma'},\OO_{X_{\Sigma'}}(D))=0$ if and only if $a\leq 0$ or $a\leq b+1.$ Indeed, by Theorem~\ref{cohp1}, $\dim_{\mathbb{C}}H^1(X_{\Sigma'},\OO_{X_{\Sigma'}}(D))$ is equal to the cardinality of the set
     \[
         \big\{ (m_1,\dots,m_n)\in\mathbb{Z}^n\mid 0<m_1+\cdots+m_n<a-b, m_1\geq -b \textnormal{ and } m_i\geq 0 \ \ \forall i=2,\dots,n   \big\}\]
        which, via the assignment $(m_1,\dots,m_n)\mapsto (m_1+b,m_2,\dots,m_n)$, is in bijection with the set
         \[\big\{ (m_1,\dots,m_n)\in\mathbb{Z}^n\mid b<m_1+\cdots+m_n<a \textnormal{ and } m_i\geq 0 \ \ \forall i=1,\dots,n   \big\}. \] This set is empty if, and only if, $a\leq 0$ or $a\leq b+1.$
\end{proof}

\section{The blow-up of projective space at several points}\label{sevpts}
 We continue with our assumption that $n\geq 3$. Let $1\leq q\leq n$ be an integer, and let $X\to\mathbb{P}^n$ be a blow-up at $q+1$ distinct points. The main result of this section, Theorem~\ref{mainthmsev}, characterizes the isomorphism classes of invertible sheaves on $X$ with vanishing cohomology in degree $1$. As a criterion for cohomological vanishing, this is well known and elementary (cf. \cite[Theorem~1.3]{DP23}), so our emphasis lies in observing that the criterion is sharp.

 Let $\Sigma$ be the fan in $N_\mathbb{R}\cong\mathbb{R}^n$ introduced in Section~\ref{onept}, whose associated toric variety is $X_{\Sigma}\cong\mathbb{P}^n$. Recall that $\Sigma$ consists of the cones generated by all proper subsets of $\{e_0,\dots,e_n\}$, so its maximal cones are the cones  \[\sigma_i=\Cone(e_0,\dots,\widehat{e}_i,\dots,e_n)\subseteq\mathbb{R}^n\] for $i=0,\dots, n$. We now let $\Sigma'$ denote the fan obtained by successive star subdivisions along the cones $\sigma_0,\dots,\sigma_q$. Then $\Sigma'$ is a refinement of $\Sigma$, and the induced toric morphism
 
 \begin{equation}\label{toricblowsev}
     \pi\colon X_{\Sigma'}\to\mathbb{P}^n
 \end{equation} is the blow-up of $\mathbb{P}^n$ at the $q+1$ distinguished points $\gamma_{\sigma_0},\dots,\gamma_{\sigma_q}$.

 \begin{rmk}
     Since $\PGL_{n+1}$ acts transitively on tuples of $q+1$ points in $\mathbb{P}^n$, making an appropriate change of coordinates shows that any blow-up $X\to\mathbb{P}^n$ at $q+1$ points in $\mathbb{P}^n$ is of the form \eqref{toricblowsev}.
 \end{rmk}

 For each $0\leq i\leq q$, let $u_i\coloneqq -e_i$, and for each $0\leq j\leq n$ with $i\neq j$, let \[\tau_{ij}\coloneqq\Cone(u_i,e_0,\dots,\widehat{e}_i,\dots, \widehat{e}_j,\dots,e_n).\] The maximal cones of $\Sigma'$ are then given by these $\tau_{ij}$ and the cones $\sigma_{q+1},\dots,\sigma_n$, while the set of ray generators of $\Sigma'$ is given by $\{u_0,\dots,u_q,e_0,\dots,e_n\}$. For each ray generator $u$, let $D_u$ denote the toric divisor in $X_{\Sigma'}$ associated to $\Cone(u)$; for example, $D_{u_0},\dots, D_{u_q}$ are the exceptional prime divisors. We now compute an explicit formula for the pullback of toric divisors in $\mathbb{P}^n$ via the blow-up morphism $\pi\colon X_{\Sigma'}\to\mathbb{P}^n$.

 \begin{lem}\label{pullbacksev}
     Let $D=\lambda_0D_0+\cdots+\lambda_nD_n$ be a toric divisor in $\mathbb{P}^n$ and let $s\coloneqq \lambda_0+\cdots+\lambda_n$. Then, the toric divisor \[D'\coloneqq\lambda_0D_{e_0}+\cdots +\lambda_n D_{e_n}+(s-\lambda_0)D_{u_0}+\cdots +(s-\lambda_q)D_{u_q}\] in $X_{\Sigma'}$ satisfies $\OO_{X_{\Sigma'}}(D')\cong\pi^*\OO_{\mathbb{P}^n}(D)$.
 \end{lem}

 \begin{proof}
     As in the proof of Lemma~\ref{pullback}, let $m_0,\dots,m_n$ be the Cartier data of $D$. We note that, for every $0\leq i\leq q$ and $0\leq j\leq n$ with $i\neq j$, the smallest cone in $\Sigma$ containing $\tau_{ij}\in\Sigma'$ is $\sigma_i$. Then, by Proposition~\ref{cartierpull}, the Cartier divisor $D''$ on $X_{\Sigma'}$ whose Cartier data is given by 
     \[ m_{\sigma_i}\coloneqq m_i\] for $i=q+1,\dots,n$, and \[m_{\tau_{ij}}\coloneqq m_i\] for $0\leq i\leq q$, $0\leq j\leq n$ ($i\neq j$), satisfies $\OO_{X_{\Sigma'}}(D'')\cong\pi^*\OO_{\mathbb{P}^n}(D)$. Direct computation now shows that $D''=D'$.
 \end{proof}

 \begin{lem}\label{pathconnsev}
    Let $D=\sum_{\rho\in\Sigma'(1)} a_\rho D_\rho$ be any toric divisor on $X_{\Sigma'}$, and let $m\in M$ be a character. If \[V_{D,m}=\bigcup_{\sigma\in\Sigma'_{\max}}\Conv(u_\rho\mid\rho\in\sigma(1), \langle m,u_\rho\rangle<-a_\rho)\subseteq\mathbb{R}^n\] is not path-connected, then $V_{D,m}=\{e_i,u_i\}$ for some $0\leq i\leq q$, or $V_{D,m}\subseteq\{u_0,\dots,u_q\}$.
\end{lem}

\begin{proof}
    Recall that $\Sigma'_{\max}=\{\sigma_{q+1},\dots,\sigma_n,\tau_{ij}\mid 0\leq i\leq q, \ 0\leq j\leq n, \ i\neq j\}$ and $\Sigma'(1)=\{\Cone(u_0),\dots,\Cone(u_q),\Cone(e_0),\dots,\Cone(e_n)\}$. Let us assume that $V_{D,m}$ is not path-connected.

    Suppose that $e_i\in V_{D,m}$ for some $0\leq i\leq n$. Since $n\geq 3$, given any ray generator $v\neq u_i$ we can find a maximal cone $\sigma\in\Sigma'_{\max}$ containing both $e_i$ and~$v$, which implies that, if $v\in V_{D,m}$, then there is a path in $V_{D,m}$ joining $e_i$ and~$v$. Since $V_{D,m}$ is assumed not to be path-connected, we deduce that $u_i\in V_{D,m}$. But then, similarly, for any ray generator $v\neq e_i$ we can find a maximal cone containing both $u_i$ and $v$, so $v\in V_{D,m}$ would imply that $V_{D,m}$ is path-connected. Therefore $V_{D,m}=\{e_i,u_i\}$.

    Now suppose that $e_0,\dots, e_q\notin V_{D,m}$. Since, for any $0\leq i,j\leq q$ with $i\neq j$, there is no maximal cone containing both $u_i$ and $u_j$, we conclude that $V_{D,m}\subseteq\{u_0,\dots,u_q\}$.
\end{proof}

Let $X\to\mathbb{P}^n$ be a blow-up at $q+1$ points, with exceptional prime divisors $E_0,\dots, E_q$. The Picard group of $X$ is generated by the classes of $E_0,\dots,E_q$ and by $\pi^*\mathcal{O}_{\mathbb{P}^n}(1)$, so the following theorem characterizes the isomorphism classes of invertible sheaves on $X$ with vanishing cohomology in degree $1$.
 
  \begin{thm}\label{mainthmsev}
    Let $a_0,\dots,a_q,b\in\mathbb{Z}$. Then,
    \[H^1\big(X,\OO_X\big(-\sum_{i=0}^qa_iE_i\big)\otimes\pi^*\OO_{\mathbb{P}^n}(b)\big)=0\] if and only if the inequalities
    \[
        a_i+a_j\leq b+1 \ \ \forall i,j\in\{0,\dots,q\}, i\neq j
    \]are satisfied, unless there is exactly one element $a_k\in\{a_0,\dots,a_q\}$ that is positive, in which case we also require $a_k\leq b+1$.
\end{thm}

\begin{proof}
    We view $\mathbb{P}^n$ as the toric variety associated to the fan $\Sigma$, and we identify $\pi$ with the toric morphism $\pi\colon X_{\Sigma'}\to\mathbb{P}^n$ by choosing coordinates so that the blown-up points match the distinguished points $\gamma_{\sigma_0},\dots,\gamma_{\sigma_q}$.

    By Lemma~\ref{pullbacksev}, we have \[\pi^*\mathcal{O}_{\mathbb{P}^n}(b)\cong\pi^*\mathcal{O}_{\mathbb{P}^n}(bD_0)\cong\mathcal{O}_{X_{\Sigma'}}(bD_{e_0}+bD_{u_1}+\cdots +bD_{u_q}),\] so we are interested in the cohomology of the divisor \[D\coloneqq bD_{e_0}-a_0D_{u_0}+(b-a_1)D_{u_1}+\cdots+(b-a_q)D_{u_q}.\] By Theorem~\ref{torcoh}, we have $H^1(X_{\Sigma'},\mathcal{O}_{X_\Sigma'}(D))=0$ if and only if $V_{D,m}$ is path-connected for every character $m\in M$. In view of Lemma~\ref{pathconnsev}, we must thus investigate under which conditions there is a character $m\in M$ such that $V_{D,m}$ is of the form $\{e_i,u_i\}$, or it is a subset of $\{u_0,\dots,u_q\}$ with at least two elements.

    Let $i\in\{1,\dots,q\}$. From the definition of the set $V_{D,m}$, and recalling that $u_j=-e_j$ for every $j\in\{0,\dots,q\}$, we see that $V_{D,m}=\{e_i,u_i\}$ if and only if the inequalities
    \begin{align*}
        a_0\leq \langle m,u_0\rangle&\leq b,\\
        b-a_i <\langle m,e_i\rangle &< 0,\\
        \langle m,e_j\rangle &\leq b-a_j \ \forall j\in\{0,\dots,q\}\setminus\{i\},\\
        0\leq \langle m,e_j\rangle & \ \ \ \ \ \ \ \ \ \ \ \  \ \forall j\in\{0,\dots,n\}\setminus\{i\}
    \end{align*} are satisfied. Notice that the above inequalities imply $b-a_i <\langle m,u_0\rangle\leq b.$ We thus find that $V_{D,m}\neq\{e_i,u_i\}$ for all $m\in M$ if and only if $a_i\leq 0$, or $a_i\leq b+1$, or $b+1\leq a_j$ for some $j\neq i$. Similarly, we find that the same statement holds for $i=0$.

    Next, let $i,j\in\{1,\dots,q\}$ with $i\neq j$. We note that $V_{D,m}$ is a subset of $\{u_0,\dots, u_q\}$ containing the elements $u_i,u_j$, if and only if the inequalities
    \begin{align*}
        \langle m,u_i\rangle &< a_i-b,\\
        \langle m,u_j\rangle &< a_j-b,\\
        \langle m,e_0\rangle &\geq -b,\\
        \langle m,e_k\rangle &\geq 0 \ \forall k\in\{1,\dots,n\}
    \end{align*} are satisfied. These inequalities admit no solution $m\in M$ if and only if $a_i+a_j\leq b+1$.

    So, suppose that $H^1(X_{\Sigma'},\mathcal{O}_{X_\Sigma'}(D))=0$, i.e., that $V_{D,m}$ is path-connected for every $m\in M$. We readily deduce that $a_i+a_j\leq b+1$ for all $i\neq j$ by the above observation, so we need only analyze the case that there is exactly one positive element $a_k$ in the set $\{a_0,\dots,a_q\}$. We then have $a_j<a_j+a_k\leq b+1$ for every $j\neq k$, so the fact that $V_{D,m}\neq\{e_k,u_k\}$ implies that $a_k\leq b+1$.

    Conversely, suppose that $a_i+a_j\leq b+1$ for all $i\neq j$ and, in case there is exactly one positive element $a_k\in\{a_0,\dots,a_q\}$, suppose also that $a_k\leq b+1$. In particular, we know that $V_{D,m}$ is not a subset of $\{u_0,\dots, u_q\}$ containing at least two elements, so it only remains to show that $V_{D,m}$ is not of the form $\{e_i,u_i\}$. We then finish our argument by swiftly analyzing the following three cases:
    \begin{enumerate}
        \item [(i)] If none of the elements in $\{a_0,\dots, a_q\}$ is positive, then the inequality $b-a_i < b$ is false for every $i=0,\dots, q$ and so $V_{D,m}$ is not of the form $\{e_i,u_i\}$.
        \item [(ii)] If there is exactly one positive element $a_k\in\{a_0,\dots,a_q\}$, then $a_i<a_i+a_k\leq b+1$, so $V_{D,m}$ is not of the form $\{e_i,u_i\}$ for $i\neq k$, and furthermore the assumption $a_k\leq b+1$ implies that $V_{D,m}\neq\{e_k,u_k\}$.
        \item [(iii)] If there are at least two positive elements in $\{a_0,\dots,a_q\}$, then $a_i<b+1$ and $V_{D,m}\neq\{e_i,u_i\}$ for every $i=0,\dots,q$.
    \end{enumerate} Therefore $H^1(X_{\Sigma'},\mathcal{O}_{X_\Sigma'}(D))=0$.
\end{proof}

\newpage



\begin{thebibliography}{99}

\bibitem[CLS11]{CLS11} D. Cox, J. Little, H. Schenck: \textit{Toric Varieties}, Graduate Studies in Mathematics, American Mathematical Society, 2011.
\bibitem[DP23]{DP23} O. Dumitrescu, E. Postinghel: \textit{Positivity of divisors on blown-up projective spaces I}, Annali della Scuola Normale Superiore di Pisa, Classe di Scienze \textbf{5} (2023), vol. 24, issue 2, 599--618.
\bibitem[Kaw82]{Kaw82} Y. Kawamata: \textit{A generalization of Kodaira–Ramanujam’s vanishing theorem}, Math. Ann. \textbf{261} (1982), no. 1, 43--46.
\bibitem[Kod53]{Kod53} K. Kodaira: \textit{On a differential-geometric method in the theory of analytic stacks}, Proc. Nat. Acad. Sci. U.S.A. \textbf{39} (1953), 1268–-1273.
\bibitem[Laz04]{Laz04} R. Lazarsfeld: \textit{Positivity in Algebraic Geometry I}, Springer-Verlag, 2004.
\bibitem[Vie82]{Vie82} E. Viehweg: \textit{Vanishing theorems}, J. Reine Angew. Math. \textbf{335} (1982), 1--8.

\end{thebibliography}
\end{document}